\documentclass[reqno]{amsart}

\usepackage{amsmath,amssymb,amsthm}

\usepackage{tikz}
\usepackage{caption}
\usepackage{graphicx}
\usepackage{graphics}

\newtheorem{theorem}{Theorem}
\newtheorem*{thm}{Theorem}
\newtheorem{lemma}{Lemma}
\newtheorem*{corollary}{Corollary}
\newtheorem*{proposition}{Proposition}

\begin{document}

\title[]{Universal lower bounds \\for Dirichlet eigenvalues}

\author[]{Stefan Steinerberger}
\address{Department of Mathematics, University of Washington, Seattle, WA 98195, USA}
\email{steinerb@uw.edu}

\begin{abstract} Let $\Omega \subset \mathbb{R}^d$ be a bounded domain and let $\lambda_1, \lambda_2, \dots$ denote the sequence of eigenvalues of the Laplacian subject to Dirichlet boundary conditions. We consider inequalities for $\lambda_n$ that are independent of the domain $\Omega$. A well--known such inequality follows from the Berezin--Li--Yau approach.
The purpose of this paper is to point out a certain degree of flexibility in the Li--Yau approach. We use it to prove a new type of two-point inequality which are strictly stronger than what is implied by Berezin-Li-Yau itself. For example, when $d=2$, one has  $ 2 \lambda_n +  \lambda_{2n} \geq  10  \pi n/|\Omega|.$
\end{abstract}

\maketitle

\section{Introduction and Results}
\subsection{Introduction.} Let $\Omega \subset \mathbb{R}^d$ be a bounded domain and consider solutions of $-\Delta \phi_n = \lambda_n \phi_n$ with Dirichlet boundary conditions $\phi_n\big|_{\Omega} = 0$. There is an infinite sequence of eigenvalues
$$ 0 < \lambda_1 < \lambda_2 \leq \lambda_3 \leq \dots$$
and Weyl's law shows that, as $n \rightarrow \infty$,
$$ \lambda_n = (1+o(1)) \cdot \frac{4 \pi^2 n^{2/d}}{ (|\Omega| \omega_d)^{2/d}} \qquad \mbox{as} \quad n \rightarrow \infty.$$
P\'olya \cite{pol0} conjectured in 1954 that this limit approximation should be one-sided and asked whether
$$ \lambda_n \geq \frac{4 \pi^2 n^{2/d}}{ (|\Omega| \omega_d)^{2/d}} \qquad \mbox{might hold for all} ~ n \in \mathbb{N}.$$
This conjecture, while verified for a number of a special domains \cite{polya}, remains difficult. It was only recently proven to hold on the Euclidean ball, see Filonov-Levitin-Polterovich-Sher \cite{fil}.
After several earlier results \cite{birman,c,met, roz1, roz2}, the currently best result follows from a universal inequality for the sum of the eigenvalues.
\begin{thm}[Li--Yau \cite{li}, 1983] One has
$$ \sum_{k=1}^{n} \lambda_k  \geq \frac{d}{d+2}  \frac{4 \pi^2 n^{1 + 2/d}}{ (|\Omega| \omega_d)^{2/d}}.$$
\end{thm}

It was later realized that the Li--Yau inequality can be obtained from an earlier 1972 result of Berezin \cite{berezin} via a Legendre transform and now these estimates are usually referred to as Berezin--Li--Yau estimates. A detailed and clear exposition can be found in the recent book by Frank-Laptev-Weidl \cite{frankl}. In particular, since the sum can be bounded from above by $n$ times its largest element $\lambda_n$, one has
$$ \lambda_n \geq  \frac{d}{d+2}  \frac{4 \pi^2 n^{2/d}}{ (|\Omega| \omega_d)^{2/d}}$$
which matches the Weyl law (and thus P\'olya's conjecture) up to the factor $d/(d+2)$. These inequalities are quite fundamental, we refer to Harrell-Hermi \cite{harr} for a discussion of the connection between Berezin--Li--Yau inequalities, the Yang inequalities \cite{ash2, yang0, yang} and the operator identity approach of Harrell-Stubbe \cite{harr2} (see also Levitin-Parnovski \cite{lev}). Starting with the work of Melas \cite{melas}, there has been a lot of interest in improved lower bounds. Melas, introducing the moment of inertia,
$$ I(\Omega) = \min_{m \in \mathbb{R}^{n}} \int_{\Omega} \|x - m\|^2 dx$$
established the following improved estimate.

\begin{thm}[Melas \cite{melas}, 2003] There exists a constant $c_d$ depending only on the dimension such that
$$ \sum_{k=1}^{n} \lambda_k  \geq \frac{d}{d+2}  \frac{4 \pi^2 n^{1 + 2/d}}{ (|\Omega| \omega_d)^{2/d}} + c_d  \frac{|\Omega|}{I(\Omega)} n.$$
\end{thm}

Several other results of this type have been obtained, we refer to Frank-Larson \cite{frank}, Geisinger-Laptev-Weidl \cite{geis},  Harrell--Stubbe \cite[Theorem 1.1]{harr3}, Harrell-Provenzano-Stubbe \cite{harr4}, Ilyin \cite{il0}, Ilyin-Laptev \cite{il}, Kovarik-Vugalter-Weidl \cite{kov}, Yolcu \cite{yol} and references therein. For more results of this flavor, we refer to the surveys of Ashbaugh \cite{ash0} and Ashbaugh-Benguria \cite{ash}.

\subsection{Main Result} The original proof of the Li--Yau inequality has a certain rigidity (discussed more extensively in \S 2.2) that suggests that some of the estimates can only be sharp sporadically. We were motivated by this observation which led to a simple and unconditional universal improvement.

\begin{theorem} One has, for all $1 \leq k \leq n$,
  $$ \lambda_n  \geq \frac{d}{d+2}  \frac{4 \pi^2 n^{2/d}}{ (|\Omega| \omega_d)^{2/d}} + \frac{k}{n}(\lambda_n - \lambda_k).$$
\end{theorem}
 The inequality has an exceedingly simple proof. While being somewhat implicit, this family of universal inequalities is, at least asymptotically, an improvement in the leading order. Let us consider, for simplicity of exposition, the case $d=2$. When $d=2$, then we know that $\lambda_n \geq 2 \pi n/|\Omega|$.
 Setting $k=n/2$ and using the Weyl law to estimate the size of the correction when $n$ becomes large, we obtain
$$ \lambda_n \geq \frac{2 \pi n}{|\Omega|}  + \frac{\lambda_n - \lambda_{n/2}}{2} \qquad \mbox{where} \qquad  \frac{\lambda_n - \lambda_{n/2}}{2} = (1+o(1)) \cdot \frac{\pi n}{|\Omega|}.$$
This shows that the size of the correction term in these inequalities can be, at least asymptotically, as large as half of the leading order term when $d=2$.

\subsection{Two-Point Inequalities} Exploiting the fact that the Li--Yau proof cannot be optimal across a range of values of $n$ (see \S 2.2), one can obtain an interesting type of improved inequality. We first illustrate a special case in two dimensions, $d=2$.  The Berezin--Li--Yau estimate then shows that
$$  \frac{\lambda_n}{n} \geq \frac{2 \pi}{|\Omega|} \qquad \mbox{and, somewhat redundantly,} \qquad  \frac{\lambda_{2n}}{2n} \geq \frac{2 \pi}{|\Omega|}.$$
We show that both inequalities cannot be close to sharp at the same time.
\begin{corollary}[Special case $d=2$]
Let $\Omega \subset \mathbb{R}^2$ and $n \in \mathbb{N}$. Then
$$     \frac{\lambda_n}{n} +  \frac{\lambda_{2n}}{2n} \geq  \frac{5  \pi}{|\Omega|}.$$
\end{corollary}

This inequality is, unsurprisingly, only a special case meant to illustrate the type of statement that is obtained.  A full description is as follows.

\begin{theorem} If the integers $n, \ell \in \mathbb{N}$ satisfy
$   (n+\ell)^{2/d} \ell \geq n^{1 + 2/d},$
then 
  \begin{align*}
 \frac{\lambda_n}{n^{2/d}} +\frac{\lambda_{n+\ell}}{(n+\ell)^{2/d}} &\geq \left(  2+  \frac{n}{\ell} \left[1 -  \left( \frac{n}{n+\ell} \right)^{2/d} \right]  \right)  \frac{d}{d+2}  \frac{4 \pi^2 n^{ 2/d}}{ (|\Omega| \omega_d)^{2/d}}.
\end{align*}
\end{theorem}
The condition ensures that $\ell$ is not too small. When $d = 2$, it requires $\ell \geq 0.618 n$, in larger dimensions $\ell$ has to be closer to $n$. The expression in the parentheses is always $>2$ and thus we always get a nontrivial inequality that is strictly better than what is implied by Berezin--Li--Yau alone. 
However, the improvement becomes smaller and less pronounced as $\ell$ becomes large compared to $n$. In particular,
$$     \frac{\lambda_n}{n^{2/d}} +  \frac{\lambda_{2n}}{(2n)^{2/d}} \geq \left(3 - \frac{1}{2^{2/d}}\right)  \frac{d}{d+2}  \frac{4 \pi^2 n^{ 2/d}}{ (|\Omega| \omega_d)^{2/d}}.$$

\subsection{An Averaging Estimate.}
The purpose of this section is to derive that a family of universal inequalities that show that the Berezin--Li--Yau estimate cannot be nearly sharp `on average'.
  Berezin--Li--Yau implies
$$ \frac{1}{n} \sum_{m=1}^{n} \frac{\lambda_m}{m^{2/d}} \geq  \frac{d}{d+2}\frac{4 \pi^2}{ (|\Omega| \omega_d)^{2/d}}.$$
\begin{theorem} There exists a constant $c_d > 0$ depending only on the dimension such that, for all $n \in \mathbb{N}$,
$$ \frac{1}{n} \sum_{m=1}^{n} \frac{\lambda_m}{m^{2/d}} \geq  (1+c_d) \frac{d}{d+2}\frac{4 \pi^2}{ (|\Omega| \omega_d)^{2/d}}.$$
 \end{theorem}
By letting $n \rightarrow \infty$, we may invoke Weyl's law to deduce that the constant $c_d$ cannot be too large and that, necessarily, $c_d \leq 2/(d+2)$. If P\'olya's conjecture is true, then the statement would be true with $c_d = 2/(d+2)$ and that would be the largest constant for which the statement holds. Our argument is optimized for brevity and is suboptimal in terms of the constant $c_d$ which was not tracked.

\section{Motivation and Proof of Theorem 1}
We summarize the idea behind the proof of the Li-Yau inequality (\S 2.1) and then explain the main idea that shows up in all the arguments in \S 2.2.
\subsection{Summary of Li--Yau}
We refer to the books of Frank-Laptev-Weidl \cite{frankl} or Levitin-Mangoubi-Polterovich \cite{levi} for additional details. 
The main idea is to use orthogonality of eigenfunctions in $L^2(\Omega)$ and the Fourier transform on $\mathbb{R}^d$ to write
$$ \sum_{m=1}^{k} \lambda_m = \sum_{m=1}^{k} \| \nabla u_m \|_{L^2(\Omega)}^2 =  \sum_{m=1}^{k} \| \xi \cdot \widehat{u_m}(\xi) \|_{L^2(\mathbb{R}^d)}^2.$$
Introducing the functions
$$ f(\xi) = \frac{|\Omega|}{(2\pi)^d} \| \xi \|^2 \qquad \mbox{and} \qquad g(\xi) = \frac{(2\pi)^d}{|\Omega|} \sum_{m=1}^{k} |\widehat{u_m}(\xi)|^2,$$
we have
$$ \sum_{m=1}^{k} \lambda_k  = \int_{\mathbb{R}^n} f(\xi) g(\xi) d\xi.$$
The final ingredient is a universal bound on $\|g\|_{L^{\infty}}$. Eigenfunctions are an orthonormal basis on $L^2(\Omega)$. Interpreting the function $e^{-i \left\langle x, \xi \right\rangle}$ as an element in  $L^2(\Omega)$, 
\begin{align*}
g(\xi) = \frac{(2\pi)^d}{|\Omega|} \sum_{m=1}^{k} \left| \left\langle (2\pi)^{-d/2} e^{-i \left\langle x, \xi \right\rangle}, u_m \right\rangle \right|^2 \leq \frac{1}{|\Omega|} \left\| e^{-i \left\langle x, \xi \right\rangle}\right\|_{L^2(\Omega)}^2 = 1.
\end{align*}
Finally, by $L^2-$normalization and Plancherel, we have
$$ \int_{\mathbb{R}^d} g(\xi) d\xi =  \int_{\mathbb{R}^d} \frac{(2\pi)^d}{|\Omega|} \sum_{m=1}^{k} |\widehat{u_m}(\xi)|^2 d\xi= \frac{(2\pi)^d}{|\Omega|}k.$$
 At this point, there is an easy way to conclude the argument: the function $g$ has a fixed $L^1-$norm
 and we are integrating over, ignoring constants, the function $\|\xi\|^2 g(\xi)$. To make the integral small, one would like to move all the $L^1-$mass of $g$ as close to the origin as possible, however, since $g(\xi) \leq 1$, we are restricted: the best way to minimize the integral is to move all the mass in a disk around the origin. Finishing the computation gives the desired bound.
 
 \subsection{A counterfactual proposition}
 This nice argument has inspired a lot of subsequent research. If P\'olya's conjecture is true, then this argument is not sharp for eigenfunctions, however, one would expect that for many orthonormal functions $u_1, \dots, u_k \in L^2(\Omega)$ this argument might be nearly sharp. However, the argument cannot be simultaneously sharp for, say, $k$ and $k+\ell$. We can encapsulate this idea in a `counterfactual' proposition (the condition is never satisfied).
 
 \begin{proposition}
 If 
 $$ \sum_{k=1}^{n} \lambda_k  = \frac{d}{d+2}  \frac{4 \pi^2 n^{1 + 2/d}}{ (|\Omega| \omega_d)^{2/d}}, \quad \mbox{then} \qquad \lambda_{n+1} \geq \frac{4 \pi^2 n^{2/d}}{(\omega_d |\Omega|)^{2/d}}.$$
 \end{proposition}
 
Thus, if Li--Yau was sharp for $\lambda_n$, we would very nearly obtain P\'olya's conjecture for $\lambda_{n+1}$ (up to lower order terms). The statement is counterfactual because Li--Yau can never be exactly sharp, however, its proof conveys an idea which can be adapted to the `almost-sharp' case. This is how our results will be obtained.
 
 \begin{proof}[Proof of the Proposition]
 If the Li--Yau inequality is sharp, then $g(\xi)$ assumes only the values 0 and 1 and is concentrated in a ball around the origin, more precisely,
 $$ g(\xi) = \frac{(2\pi)^d}{|\Omega|} \sum_{k=1}^{n} |\widehat{u_k}(\xi)|^2 = 1_{B(0,r)},$$
 where $r = 2 \pi n^{1/d}/ (\omega_d |\Omega|)^{1/d}$. 
 This means that $\widehat{u}_{n+1}(\xi)$ has be identically 0 inside the ball with that radius and has to have its support fully outside the ball so as to not violate the inequality $g(\xi) \leq 1$. Then, however,
 $$ \lambda_{n+1} = \int_{\mathbb{R}^d} \|\xi\|^2 |\widehat{u_{n+1}}(\xi)|^2 d\xi \geq r^2 \int_{\mathbb{R}^d}   |\widehat{u_{n+1}}(\xi)|^2 d\xi  = r^2.$$
 \end{proof}
 
 In practice, `near-optimality' of Li--Yau would translate into $g$ being very nearly $1_{B(0,r)}$ which then forces `most' of the $L^2-$mass of subsequent eigenfunctions to be outside that ball. One way of making this precise will be illustrated in \S 2.2.
 
 \subsection{A Mass Concentration Version.}
  The purpose of this section is to re-prove the Li--Yau estimate in a formulation that comes with a parameter $\eta$ that measures what fraction of the $L^1-$mass of $g$ lies within the optimal ball. In the worst case ($\eta = 0$), we recover the Li--Yau estimate. Conversely, if Li--Yau was close to sharp, $\eta$ has to be close to 0 which then means that most of the mass lies inside the ball.
 \begin{lemma}[Parametrized Li--Yau] Consider the eigenfunctions $u_1, \dots, u_k$. There exists a real number $0 < \eta \leq 1$ such that
$$ \int_{B\left(0,  \frac{ 2 \pi  k^{1/d}}{ (\omega_d  |\Omega|)^{1/d}}\right)}  \sum_{m=1}^{k} |\widehat{u_m}(\xi)|^2 d\xi = (1-\eta) k.$$
We have
$$ \sum_{m=1}^{k} \lambda_m \geq  \underbrace{\left( (1-\eta)^{(d+2)/d} + \frac{d+2}{d} \eta\right)}_{\geq 1} \frac{d}{d+2}  \frac{4 \pi^2 k^{1+2/d}}{ (|\Omega| \omega_d)^{2/d}}.$$
 \end{lemma}

 \begin{proof}
 We follow the notation and outline of the argument given above. Using
$$\sum_{m=1}^{k} \lambda_m = \sum_{m=1}^{k} \| \xi \widehat{u_m} \|_{L^2(\mathbb{R}^d)} = \int_{\mathbb{R}^d}  \frac{|\Omega|}{(2\pi)^d} |\xi|^2 \cdot \frac{(2\pi)^d}{|\Omega|} \sum_{m=1}^{k} |\widehat{u_m}(\xi)|^2 d\xi$$
and the functions $f(\xi)$ and $g(\xi)$ as above, one has
\begin{align*}
g(\xi) = \frac{(2\pi)^d}{|\Omega|} \sum_{m=1}^{k} \left| \left\langle (2\pi)^{-d/2} e^{-i \left\langle x, \xi \right\rangle}, u_m \right\rangle \right|^2 \leq \frac{1}{|\Omega|} \left\| e^{-i \left\langle x, \xi \right\rangle}\right\|_{L^2(\Omega)}^2 = 1
\end{align*} 
 as well as
$$ \int_{\mathbb{R}^d} g(\xi) d\xi = \frac{(2\pi)^d k}{|\Omega|}.$$
 We also introduce the radius  $ r = 2 \pi k^{1/d} /(\omega_d |\Omega|)^{1/d}$
 which is chosen so that
 $$ \int_{B(0,r)} 1 ~dx = \int_{\mathbb{R}^d} g(\xi) d\xi.$$
 Since $0 \leq g(\xi) \leq 1$, this implies the existence of a $0 \leq \eta \leq 1$ so that
 $$  \int_{B(0,r)} g(\xi) dx = (1-\eta) \int_{\mathbb{R}^d} g(\xi) d\xi =(1-\eta)  \frac{(2\pi)^d k}{|\Omega|}.$$
 At this point we deviate from the classical argument by applying the classic bathtub principle twice. If we know a priori that
 the $L^1-$mass of $g$ contained in the ball is only $(1-\eta)$ of the full capacity of the ball, then the worst case is, again,
 when the mass is concentrated: this leads to the old computation with the radius scaled down by $(1-\eta)^{1/d}$ and we get
 \begin{align*}
    \int_{B(0,r)} f(\xi) g(\xi) dx &\geq \frac{|\Omega|}{(2\pi)^d} \int_{B(0, (1-\eta)^{1/d} r)} |\xi|^2 d\xi \\
   &=  \frac{|\Omega|}{(2\pi)^d} \frac{r^{d+2} (1-\eta)^{(d+2)/d} }{d+2} \sigma_{d-1} \\
   &= \frac{4 \pi^2 k^{1+2/d}}{ (|\Omega| \omega_d)^{2/d}} \cdot \frac{d}{d+2} \cdot (1-\eta)^{(d+2)/d}.
 \end{align*}
 However, the remainder of the $L^1-$mass of $g$ has to be somewhere outside of $B(0,r)$. In light of $f$ being radial and monotonically increasing, applying the bathtub principle one more time shows that the worst case is if the mass is arranged just outside $B(0,r)$ in the shape of a spherical shell $B(0,s) \setminus B(0,r)$. We start by determining $s$. Volume considerations lead to
 $$  \omega_d (s^d - r^d) = \eta \frac{(2\pi)^d k}{|\Omega|} \qquad \mbox{giving} \qquad s^d = \frac{\eta}{\omega_d} \frac{(2\pi)^d k}{|\Omega|} + r^d.$$
 Thus, in the worst case,
 \begin{align*}
    \int_{\mathbb{R}^d \setminus B(0,r)} f(\xi) g(\xi) dx &\geq \frac{|\Omega|}{(2\pi)^d} \int_{r}^s t^2  \sigma_{d-1} t^{d-1} dt \\
   &=  \frac{|\Omega|}{(2\pi)^d} \frac{\sigma_{d-1}}{d+2} \left( s^{d+2} - r^{d+2} \right).
    \end{align*}
 To get this into a nicer shape, we use the inequality, valid for $x \geq y \geq 0$ that
 $$ x^{(d+2)/d} - y ^{(d+2)/d} \geq \frac{d+2}{d} y^{2/d} (x-y).$$
 
 This yields
  \begin{align*}
    \int_{\mathbb{R}^d \setminus B(0,r)} f(\xi) g(\xi) dx &\geq \frac{|\Omega|}{(2\pi)^d} \frac{\sigma_{d-1}}{d+2} \left( \left(  \frac{\eta}{\omega_d} \frac{(2\pi)^d k}{|\Omega|} + r^d \right)^{(d+2)/d} - (r^{d})^{(d+2)/d} \right) \\
    &\geq \frac{|\Omega|}{(2\pi)^d} \frac{\sigma_{d-1}}{d+2} \frac{d+2}{d} r^2  \frac{\eta}{\omega_d} \frac{(2\pi)^d k}{|\Omega|} \\
    &= \eta    r^2    k = \eta k  4 \pi^2 \left( \frac{k}{\omega_d |\Omega|}\right)^{2/d}
    \end{align*}
 Altogether, we have
\begin{align*}
   \int_{\mathbb{R}^d} f(\xi) g(\xi) d\xi  &\geq \frac{4 \pi^2 k^{1+2/d}}{ (|\Omega| \omega_d)^{2/d}} \cdot \frac{d}{d+2} \cdot (1-\eta)^{(d+2)/d} +  \eta     \frac{ 4 \pi^2 k^{1 + 2/d}}{ (\omega_d |\Omega| )^{2/d}} \\
 &= \frac{d}{d+2}  \frac{4 \pi^2 k^{1+2/d}}{ (|\Omega| \omega_d)^{2/d}} \left( (1-\eta)^{(d+2)/d} + \frac{d+2}{d} \eta\right).
 \end{align*}
 We note that this may be interpretable as a quantitative improvement since
 $$   (1-\eta)^{(d+2)/d} + \frac{d+2}{d} \eta \geq 1 \qquad \mbox{for all}~0 \leq \eta \leq 1.$$
In particular, we also get
$$ \frac{\lambda_k}{k} \geq  \frac{d}{d+2}  \frac{4 \pi^2 k^{2/d}}{ (|\Omega| \omega_d)^{2/d}} \left( (1-\eta)^{(d+2)/d} + \frac{d+2}{d} \eta\right).$$
 \end{proof}

\subsection{Proof of Theorem 1}
It is now clear how one would like to apply these ideas: if the Li--Yau argument was nearly sharp for $\lambda_k$, one would be able to apply the preceding Lemma 1 with a small value of $\eta$ and then force subsequent Fourier transform to be mostly localized outside the ball corresponding to $k$. Using the bathtub principle a second time would then lead to an improvement. This argument, when written out and using 
$  \lambda_m = \| \nabla u_m \|_{L^2(\Omega)}^2 $
to control each contribution has an elementary that uses nothing at all.
\begin{proof}
We use only the Li--Yau estimate. For $0 \leq k \leq n-1$
\begin{align*}
\lambda_n &\geq \frac{1}{n-k} \sum_{m=k+1}^{n} \lambda_m =  \frac{1}{n-k} \left(\sum_{m=1}^{n} \lambda_m - \sum_{m=1}^{k} \lambda_m \right) \\
&\geq  \frac{1}{n-k} \left(\sum_{m=1}^{n} \lambda_m - k \cdot \lambda_k \right) \geq \frac{1}{n-k} \left(  \frac{d}{d+2}  \frac{4 \pi^2 n^{1+2/d}}{ (|\Omega| \omega_d)^{2/d}}  - k \cdot \lambda_k \right).
\end{align*} 
Rearranging shows that
$$ \frac{n-k}{n} \lambda_n + \frac{k}{n} \lambda_k \geq  \frac{d}{d+2}  \frac{4 \pi^2 n^{2/d}}{ (|\Omega| \omega_d)^{2/d}} $$
and from this the result follows. The final inequality remains valid, though it does not say anything new, when $k=n$.
\end{proof}

\section{Proof of Theorem 2}
\begin{proof}
We abbreviate
$$ c_d = \frac{d}{d+2}  \frac{4 \pi^2}{ (|\Omega| \omega_d)^{2/d}}$$
for ease of exposition.
Suppose that 
$$ \lambda_n = (1+\eta) c_d n^{2/d},$$
where we think of $\eta \geq 0$ as a small quantity.
Then, trivially,
$$ \sum_{m=1}^{n} \lambda_m \leq n \cdot \lambda_n \leq (1+\eta)c_d n^{1+2/d}.$$
We then argue that
\begin{align*}
\lambda_{n+\ell} &\geq \frac{1}{\ell} \sum_{m=n+1}^{n + \ell} \lambda_m =  \frac{1}{\ell} \left(\sum_{m=1}^{n+\ell} \lambda_m - \sum_{m=1}^{n} \lambda_m \right) \\
&\geq \frac{1}{\ell} \left(\sum_{m=1}^{n+\ell} \lambda_m  - (1+\eta)c_d n^{1+2/d} \right) \\
&\geq \frac{1}{\ell} \left(c_d (n+\ell)^{ 1+2/d}  -(1+\eta) c_d n^{1+ 2/d}  \right) \\
&=c_d \cdot (n+\ell)^{2/d} \cdot \frac{ (n+\ell) - (1+\eta) \frac{n^{1+2/d}}{(n+\ell)^{2/d}} }{\ell}.
\end{align*} 
Therefore
\begin{align*}
\frac{1}{c_d} \left(\frac{\lambda_n}{n^{2/d}} +\frac{\lambda_{n+\ell}}{(n+\ell)^{2/d}}\right) &= 1+ \eta+  \frac{ (n+\ell) - (1+\eta) \left( \frac{n}{n+\ell} \right)^{2/d} n }{\ell}.
\end{align*}
This expression is linear in $\eta$. If the coefficient in front of $\eta$ is non-negative, then the minimum is attained when $\eta =0$. Thus, we require that the coefficient satisfies
$$ 1 - \left( \frac{n}{n+\ell} \right)^{2/d} \frac{n}{\ell} \geq 0 \qquad \mbox{or} \qquad (n+\ell)^{2/d} \ell \geq n^{1 + 2/d}$$
  which, in particular, is always true when $\ell = n$. If the condition is satisfied, one obtains the lower bound by plugging in $\eta = 0$ to obtain
  \begin{align*}
\frac{1}{c_d} \left(\frac{\lambda_n}{n^{2/d}} +\frac{\lambda_{n+\ell}}{(n+\ell)^{2/d}}\right) &\geq 2+  \frac{n}{\ell}\left(1 -  \left( \frac{n}{n+\ell} \right)^{2/d}\right).
\end{align*}
 \end{proof}
 
 \section{Proof of Theorem 3}
  We start by explicitly verifying the case $n=1$. We know a priori that this inequality has to be true and show something a little bit more explicit: that there is an associated gap when $n=1$. This can then be leveraged into a weaker form of Theorem 3 where the constant decays with $n$.  This is far from the desired uniform bound but it does allow us to assume that $n \geq n_0$ is sufficiently large.
 \begin{lemma} We have
 $$ \frac{1}{n} \sum_{m=1}^{n} \frac{\lambda_m}{m^{2/d}}\geq  \left(1+ \frac{1}{n}\right) \cdot\frac{d}{d+2}\frac{4 \pi^2}{ (|\Omega| \omega_d)^{2/d}}.$$
 \end{lemma}
 \begin{proof}
 The Faber-Krahn inequality yields
 $$ \lambda_1(\Omega) \geq \frac{\pi}{\Gamma(d/2 + 1)^{2/d}} \frac{j_{d/2-1,1}^2}{|\Omega|^{2/d}}.$$
To verify that this is consistent with the Berezin--Li--Yau statement, we need
  $$  \frac{4d}{d+2} <  j_{d/2-1,1}^2.$$
This is easily verified for small $d \in \mathbb{N}$ and then with standard asymptotics for large $d$. In particular, the ratio between the quantities is minimized when $d=2$ (where it is $\sim 2.89$) and therefore we always have, when $n=1$,
$$ \sum_{m=1}^{1} \frac{\lambda_m}{m^{2/d}} \geq  \left(1+1.8\right) \cdot\frac{d}{d+2}\frac{4 \pi^2}{ (|\Omega| \omega_d)^{2/d}}.$$
Using Berezin--Li--Yau on all the other terms then shows 
$$ \frac{1}{n} \sum_{m=1}^{n} \frac{\lambda_m}{m^{2/d}}\geq  \left(1+ \frac{1}{n}\right) \cdot\frac{d}{d+2}\frac{4 \pi^2}{ (|\Omega| \omega_d)^{2/d}}.$$
\end{proof}

\begin{proof}[Proof of Theorem 3.]
For simplicity of exposition we shall again abbreviate
 $$ c_d =  \frac{d}{d+2}\frac{4 \pi^2}{ (|\Omega| \omega_d)^{2/d}}.$$
Lemma 1 allows to assume that $n$ is sufficiently large. The argument presented below works for $n$ sufficiently large (depending only on $d$). Thus, for small $n$ we can use Lemma 2 while for $n$ sufficiently large, we use the argument below. Assume now that $n$ is large and that 
 $$ \frac{1}{n} \sum_{m=1}^{n} \frac{\lambda_m}{m^{2/d}} \leq  (1+\varepsilon) c_d.$$
The goal is to deduce a contradiction once $\varepsilon$ is sufficiently small. By pigeonholing, there exists an integer
$ n/3 \leq \ell \leq 2n/3$
such that
$$\frac{\lambda_{\ell}}{\ell^{2/d}} \leq (1 + 4 \varepsilon) c_d$$
since one could otherwise argue that
\begin{align*}
 \frac{1}{n} \sum_{m=1}^{n} \frac{\lambda_m}{m^{2/d}} \geq  \frac{1}{n} \sum_{m=1}^{n} c_d\left(1 + 1_{n/3 \leq m \leq 2n/3} \cdot 4 \varepsilon\right) \geq (1+\varepsilon) c_d. 
 \end{align*}
 In particular, it implies
 $$ \sum_{m=1}^{\ell} \lambda_m \leq \ell \cdot \lambda_{\ell} \leq (1 + 4 \varepsilon) c_d \ell^{1 + 2/d}.$$
 Let now $k \geq \ell +1$. Then, arguing as above,
\begin{align*}
\lambda_{k} &\geq \frac{1}{k-\ell} \sum_{m=\ell+1}^{k} \lambda_m =  \frac{1}{k-\ell} \left(\sum_{m=1}^{k} \lambda_m - \sum_{m=1}^{\ell} \lambda_m \right) \\
&\geq  c_d \cdot k^{2/d} \cdot \frac{ k- (1+4 \varepsilon) \frac{\ell^{1+2/d}}{k^{2/d}} }{k-\ell}.
\end{align*} 
We conclude the argument by considering 
$$ \sum_{k=3n/4}^{n} \frac{\lambda_{k}}{k^{2/d}} \geq c_d \sum_{k=3n/4}^{n} \max\left\{1,  \frac{ k- (1+4 \varepsilon) \frac{\ell^{1+2/d}}{k^{2/d}} }{k-\ell} \right\}.$$
Since $\ell \leq 2n/3 \leq 3n/4 \leq k$, we have
$$  \frac{\ell^{1+2/d}}{k^{2/d}} \leq \left( \frac{8}{9} \right)^{2/d} \ell$$
and thus
$$ \sum_{k=3n/4}^{n} \frac{\lambda_{k}}{k^{2/d}} \geq c_d\sum_{k=3n/4}^{n} \max\left\{1,  \frac{ k- (1+4 \varepsilon) \left(\frac{8}{9}\right)^{2/d} \ell  }{k-\ell} \right\}.$$
For any $ \varepsilon < \varepsilon_0$ where $\varepsilon_0$ is chosen such that
$$ (1+4 \varepsilon_0) \left(\frac{8}{9}\right)^{2/d} = \frac{1 + (8/9)^{2/d}}{2} < 1$$
we then see that, for some $\varepsilon_1 > 0$ depending only on $d$,
$$ \sum_{k=3n/4}^{n}   \frac{ k- (1+4 \varepsilon) \left(\frac{8}{9}\right)^{2/d} \ell  }{k-\ell}  \geq (1+ \varepsilon_1) \frac{n}{4}$$
which concludes the argument (where the final constant can be chosen as $\varepsilon_1/8$).
\end{proof}


\begin{thebibliography}{10}

\bibitem{ash0} M. Ashbaugh, Isoperimetric and universal inequalities for eigenvalues. Spectral theory and geometry (Edinburgh, 1998), 273, p. 95--139.

\bibitem{ash2} M. Ashbaugh, The universal eigenvalue bounds of Payne-Polya-Weinberger, Hile-Protter, and HC Yang,  In Proceedings of the Indian Academy of Sciences-Mathematical Sciences, vol. 112, no. 1, pp. 3-30. New Delhi: Springer India, 2002.

\bibitem{ash} M. Ashbaugh and R. Benguria,  Isoperimetric inequalities for eigenvalues of the Laplacian. In Proceedings of Symposia in Pure Mathematics (Vol. 76, No. 1, p. 105). Providence, RI; American Mathematical Society; 1998.

\bibitem{berezin} F. A. Berezin, Covariant and contravariant symbols of operators, Izv. Akad. Nauk SSSR Ser. Mat. 13 (1972), p. 1134--1167.

\bibitem{birman} M. Sh. Birman and M. Z. Solomyak, The principal term of the spectral asymptotics formula
for “non-smooth” elliptic problems, Functional Anal. Appl. 4 (1970), p. 265--275.

\bibitem{yang0} D. Chen, T. Zheng and H. Yang, Estimates of the gaps between consecutive eigenvalues of Laplacian. Pacific Journal of Mathematics, 282 (2016), p. 293--311.

\bibitem{c} Z. Ciesielski, On the spectrum of the Laplace operator, Comment. Math. Prace Mat. 14
(1970), p. 41--50.


\bibitem{fil} N. Filonov, M. Levitin, I. Polterovich and D. Sher, P\'olya’s conjecture for Euclidean balls. Inventiones mathematicae, 234 (2023), p. 129--169.

\bibitem{frank} R. Frank and S. Larson, Two-term spectral asymptotics for the Dirichlet Laplacian in a Lipschitz domain. Journal für die reine und angewandte Mathematik (Crelles Journal) 766 (2020), p.195--228.

\bibitem{frankl} R. Frank, A. Laptev, T. Weidl,  Schr\"odinger operators: eigenvalues and Lieb–Thirring inequalities (Vol. 200). Cambridge University Press, 2022.

\bibitem{geis} L. Geisinger,  A. Laptev, and T. Weidl, Geometrical versions of improved Berezin–Li–Yau inequalities, Journal of Spectral Theory 1, no. 1 (2011): p. 87--109.

\bibitem{harr} E. Harrell II and L. Hermi, On Riesz means of eigenvalues, Communications in Partial Differential Equations 36, no. 9 (2011): p. 1521--1543.

\bibitem{harr2}  E. Harrell and J. Stubbe, On trace identities and universal eigenvalue
estimates for some partial differential operators. Trans. Amer. Math. Soc.
349: p. 1797--1809.

\bibitem{harr3}  E. Harrell and J. Stubbe, Two-term, asymptotically sharp estimates for eigenvalue means of the Laplacian, J. Spectral Theory 8 (2018), p. 1529--1550

\bibitem{harr4} E.  Harrell, L. Provenzano and J. Stubbe,  Complementary asymptotically sharp estimates for eigenvalue means of Laplacians. International Mathematics Research Notices, 11 (2021), p. 8405--8450.

\bibitem{il0} A. Ilyin, Lower bounds for sums of eigenvalues of elliptic operators and systems. Sbornik: Mathematics, 204 (2013), 563.

\bibitem{il} A. Ilyin and A. Laptev,  Berezin–Li–Yau inequalities on domains on the sphere. Journal of Mathematical Analysis and Applications, 473 (2019), p. 1253--1269.

\bibitem{laptev} A. Laptev, Dirichlet and Neumann eigenvalue problems on domains in Euclidean spaces. J. Funct. Anal. 151, p. 531--545 (1997)

\bibitem{levi} M. Levitin, D. Mangoubi and I. Polterovich, Topics in spectral geometry (Vol. 237). American Mathematical Society, 2023.

\bibitem{lev} M. Levitin and L. Parnovski, Commutators, spectral trace identities, and universal estimates for eigenvalues. Journal of Functional Analysis, 192 (2002), p. 425--445.

\bibitem{li} P. Li and S. T. Yau, On the Schr\"odinger equation and the eigenvalue problem, Comm. Math. Phys. 88, p. 309--318 (1983)

\bibitem{kov} H. Kovarik, S. Vugalter, T. Weidl, Two-dimensional Berezin--Li--Yau inequalities with a correction term. Comm. Math. Phys. 3 (2009), p. 959--981

\bibitem{melas} A. D. Melas, A lower bound for sums of eigenvalues of the Laplacian. Proc. Amer. Math. Soc. 131, p. 631--636 (2003)

\bibitem{met} G. Metivier, Valeurs propres de probl`emes aux limites elliptiques irr´eguliers, Bull. Soc.
Math. France, Mem. 51–52 (1977), p. 125--229.

\bibitem{pol0}  G. Polya,  Mathematics and plausible reasoning. In: Two Volumes. Princeton University Press, Princeton (1954).

\bibitem{polya} G. Polya, On the eigenvalues of vibrating membranes. Proc. London Math. Soc. 11, p. 419--433 (1961)

\bibitem{roz1} G. V. Rozenblum, The distribution of eigenvalues of the first boundary value problem in
unbounded regions, Soviet Math. Dokl. 12 (1971), p. 1539--1542

\bibitem{roz2} G. V. Rozenblum, On the eigenvalues of the first boundary value problem in unbounded
regions, Math. USSR-Sb. 18 (1972), p. 235--248 

\bibitem{yang} H. C. Yang, Estimates of the difference between consecutive eigenvalues.
(Revision of International Centre for Theoretical Physics preprint IC/91/60, 1995

\bibitem{yol} S. Yolcu, An improvement to a Berezin--Li--Yau type inequality, Proceedings of the American Mathematical Society 138, no. 11 (2010): p. 4059--4066.

\end{thebibliography}
\end{document}